\documentclass[12pt]{article}

\usepackage{tikz-cd}
\usepackage{mymacro}
\usepackage{typearea}
\usepackage{cleveref}
\usepackage{url}

\renewcommand{\Mod}{\mathrm{Mod}}

\newcommand{\Perf}{\mathrm{Perf}}

\renewcommand{\Coh}{\mathrm{Coh}}

\title{Adjoints, wrapping, and morphisms at infinity}
\author{Tatsuki Kuwagaki and Vivek Shende}
\date{\today}

\begin{document}

\maketitle
\begin{abstract}
For a localization of a smooth proper category along a subcategory preserved by the Serre functor, we show that morphisms in Efimov's algebraizable categorical formal punctured neighborhood of infinity can be computed using the natural cone between right and left adjoints of the localization functor. In particular, this recovers the following result of Ganatra--Gao--Venkatesh: morphisms in categorical formal punctured neighborhoods of wrapped Fukaya categories are computed by Rabinowitz wrapping.
\end{abstract}

To any dg category $\mathcal{S}$ over a field $\bK$, 
Efimov has associated an ``algebraizable 
categorical formal punctured neighborhood of infinity" \cite{Efimov}:
$$\cS \to \widehat{\cS}_\infty$$
We are interested here in the case when $\cS$ admits
a localization sequence 
\begin{equation} \label{seq} 
    0 \to \cK \xrightarrow{j} \cC \xrightarrow{i^L} \cS \to 0
\end{equation}
where $\cC$ is smooth (perfect diagonal bimodule) and locally proper (finite dimensional Hom spaces). 

In this case, Efimov showed that $\widehat{\cS}_\infty$ can be computed
as follows.  To any category $\cT$ we may associate its `pseudo-perfect
modules' $\cT^{pp} = \Hom(\cT, \Perf\, (\bK))$.  
Since $\cK$ is locally proper, the Yoneda embedding gives
$\cK \hookrightarrow  \cK^{pp}$.  Form the quotient: 
\begin{equation} \mathrm{Perf}_{top}(\widehat{\cS}_\infty):= \cK^{pp} /\cK
\end{equation}
The composition of the Yoneda functor with passage to 
the quotient gives a map 
$$\mathcal{C} \to \Hom(\cK, \mathrm{Perf}(\bK))/\cK$$
This map evidently factors through $\cS$, and 
$\widehat{\cS}_\infty$ 
is
the full subcategory of 
$\mathrm{Perf}_{top}(\widehat{\cS}_\infty)$ 
generated by the image of  
$\cS$, or equivalently $\cC$.  

As always with quotient categories, 
it is not easy to compute morphism spaces directly from the definition. 
Our purpose here is to give a more explicit formula for morphisms in $\widehat{\cS}_\infty$, under the additional assumption that 
the Serre functor of $\cC$ preserves $\cK$.
Our result  
is inspired by, and implies, a result of 
Gao-Ganatra-Venkatesh in the situation where $\cS$ is the Fukaya category of a Weinstein manifold \cite{GGV}.

\begin{theorem} \label{rab formula}
Assume given a sequence as in \eqref{seq}, such that $\cC$ is smooth and locally proper, and the Serre functor of $\cC$ preserves $\cK$. 
Let $i:\Mod\, \cS \to \Mod\, \cC$ be 
the pullback functor on module categories. 
Then for $c, d \in \cC$, there
is a natural isomorphism
$$\Hom_{\widehat \cS_\infty}(c, d) 
= \Cone(\Hom_{\Mod\, \cC}(ii^L(c), d)\rightarrow \Hom_{\Mod\, \cC}(c, ii^L(d)))$$
where the map is induced by the 
unit maps 
$c \to ii^L(c)$ and $d \to ii^L(d)$. 
\end{theorem}
\begin{remark}
The map $i$ also has a right adjoint $i^R$; 
 we can also express the formula as 
$\Hom_{\widehat \cS_\infty}(c, d) = \Hom_{\Mod\, \cC}(c, \Cone(ii^R(d) \to ii^L(d)))$. 
\end{remark}
\begin{remark}
It may be nontrivial to express
compositions in $\widehat{S}_\infty$ in terms
of the formula above.  We give an expression at the level of cohomology in Appendix \ref{compositions}.
\end{remark}

We will give the proof of this theorem after illustrating in algebraic and symplectic geometry: 

\begin{example}[Coherent sheaves]
Let $Y$ be a smooth proper
algebraic variety, and $X \subset Y$ an open subvariety with complement $Z$. Then 
$\Coh(Y)$ is smooth and proper, 
and one has 
$$\Coh(X) = \Coh(Y) / \Coh_Z(Y),$$ where $\Coh_Z(Y)$ is the 
full subcategory on sheaves set-theoretically supported on $Z$. The Serre functor of $\Coh(Y)$ obviously preserves $\Coh_Z(Y)$.
Writing $x: X \to Y$ for the inclusion, 
our result asserts that 
given $E, F \in \Coh(Y)$, 
$$\Hom_{\widehat{\Coh(X)}_\infty}(E, F) = \Cone(\Hom_{Q \Coh(Y)}(x_* x^* E, F) \to \Hom_{Q \Coh(Y)}(E, x_* x^* F))$$
Note we may compute this cone of Homs
after restricting to any Zariski
neighborhood of $Z$, since $x_* x^*E \to E$
and $x_* x^*F \to F$
are isomorphisms away from such neighborhood. 

Let us do an example of 
the example.  We take $Y = \bP^1$, 
$X = \bP^1 \setminus 0$, and $E = F = \mathcal{O}$.  In the Zariski
chart $\bP^1 \setminus \infty$,  we  compute:  

$$\Cone(\Hom_{\bK[t]}(\bK[t, t^{-1}], 
\bK[t]) \to \Hom_{\bK[t]}(\bK[t], \bK[t, t^{-1}])) \cong \bK((t))$$

Indeed, the second term in the cone 
is obviously $\bK[t, t^{-1}]$.  
One can show that the first is in fact isomorphic to 
 $(\bK[[t]]/\bK[t]) [-1]$; 
  we include a calculation in Appendix \ref{ridiculous hom calculation}.  We leave it to the reader to check that 
 the cone realizes the nontrivial extension
 $$ 0 \to \bK[t, t^{-1}] \to \bK((t)) \to \bK[[t]]/\bK[t] \to 0.$$
\end{example}


Before going to the next example, we note the following lemma.
\begin{lemma}\label{lem:LeftCYandSerre}
Let $\cB$ be a smooth proper category, and $f\colon \cA\rightarrow \cB$ be a left relative Calabi--Yau structure. Then the Serre functor of $\cB$ preserves the image of the triangulated hull of $\cA$.
\end{lemma}
\begin{proof}
By the definition of left relative CY structure~\cite[(4.10)]{BravDyckerhoff}, we have an exact sequence
$$ \cB^!\rightarrow f_!\cA[-n+1]\rightarrow \cB[-n+1]$$
in the category of $\cB^e=\cB\otimes_k\cB$-bimodules. Here $\cB^!$ is the dualizing module, $n$ is the dimension of the CY structure. Since $\cB$ is smooth proper, $\otimes \cB^!$ is the inverse Serre functor. As a result, the sequence says that the Serre functor differs from the identity by something from $\cA$. Hence the images of objects of $\cA$ are preserved by the Serre functor of $\cB$.
\end{proof}

\begin{example}[Fukaya categories]
Let $W$ be a Weinstein symplectic manifold and 
$\Lambda \subset \partial_\infty W$ a generically Legendrian 
total stop, such as the core of a fiber of an open book decomposition of $\partial_\infty W$, e.g. as for a Lefschetz fibration on $W$ in the sense of Seidel. 
Then \cite{GPS2}
the (partially) wrapped Fukaya category 
$\mathrm{Fuk}(W, \Lambda)$ is smooth and proper, 
and we have a localization sequence 
\begin{equation} \label{localization-fuk}
0 \to \langle D_{\Lambda} \rangle \to \mathrm{Fuk}(W, \Lambda) \to \mathrm{Fuk}(W) \to 0
\end{equation}
where $D_{\Lambda}$ are the so-called linking disks to $\Lambda$. 

We claim that $\langle D_\Lambda \rangle$ is preserved by the Serre functor of $\mathrm{Fuk}(W, \Lambda)$.  In the case where $\Lambda$ is the core of fiber of a Lefschetz fibration, this follows from Seidel's result that the (inverse) Serre functor on $\mathrm{Fuk}(W, \Lambda)$ acts  by ``wrapping once'' \cite{Seidel-SHasHH}, which evidently preserves the $D_\Lambda$.\footnote{In \cite{Seidel-SHasHH} this was asserted as a conjecture, and was proved in \cite[Eq. (7.63)]{Seidel-LefschetzII} for Seidel's definition of Fukaya--Seidel categories.  Strictly speaking, Seidel's setup differs from that in \cite{GPS2} in terms of the asymptotic conditions imposed at infinity; a detailed account of the isotopies needed to check that the approaches are equivalent
can be found in \cite{Jeffs}.}  Note that any Weinstein manifold can be presented as a Lefschetz fibration \cite{Giroux-Pardon}, hence equipped with such a stop. 

More generally, for any (say Whitney stratifiable) Legendrian total stop $\Lambda$, while we do not know an explicit description of the Serre functor of $\mathrm{Fuk}(W, \Lambda)$, we can nevertheless check that it preserves $\langle D_\Lambda \rangle$.  First we interpret  $\mathrm{Fuk}(W, \Lambda)$ with a category of microsheaves $\mathrm{\mu sh}(W,\Lambda)$ through \cite{GPS3}. We have a left relative CY structure on $\mathrm{\mu sh}(\Lambda) \to \mathrm{\mu sh}(W,\Lambda)$ by \cite{Shende-Takeda}. Then Lemma~\ref{lem:LeftCYandSerre} implies what we want.

Thus, we may apply Theorem \ref{rab formula}.  Let us see what it yields.
Suppose given a Lagrangian $M \in \mathrm{Fuk}(W, \Lambda)$.  As in \cite{GPS2}, by a  {\em negative wrapping} $M \rightsquigarrow M^-$, we mean 
an isotopy induced by a Hamiltonian 
which is linear and negative at contact infinity.  So long as $M^-$ avoids $\Lambda$ and hence defines an element of $\mathrm{Fuk}(W, \Lambda)$, there is a continuation morphism $M \to M^-$. 
Essentially by definition,\footnote{Or see \cite[Lemma 3.12]{GPS2} for a detailed argument in the equivalent version where the wrapping is done positively in the first factor, rather than negatively in the second.} 
$$\Hom_{\mathrm{Fuk}(W)}(\,\cdot\, , M) = 
\lim_{\substack{\longrightarrow \\ M \to M^-}} \!\!
\Hom_{\mathrm{Fuk}(W, \Lambda)}(\,\cdot\, , M^-)
= 
\Hom_{\Mod \, \mathrm{Fuk}(W)}(\,\cdot\,, \!\! \lim_{\substack{\longrightarrow \\ M \to M^-}} \!\! M^-).$$

In other words, there is a natural isomorphism 
$$ii^L(M) \cong \lim_{\substack{\longrightarrow \\ M \to M^-}} \!\!  M^-$$
We conclude:

\begin{eqnarray*}
    \Hom_{\widehat{\mathrm{Fuk}(W, \Lambda')}_\infty } (L, M) & = & \Cone(\Hom_{\mathrm{Fuk}(W, \Lambda)}(\!\lim_{\substack{\longrightarrow \\ L \to L^-}} \!\! L^-, M) \to \Hom_{\mathrm{Fuk}(W, \Lambda)}(L, \!\!\lim_{\substack{\longrightarrow \\ M \to M^-}} \!\! M^-)  ) \\
& = &
\Cone(\!\lim_{\substack{\longleftarrow \\ L \to L^-}} \Hom_{\mathrm{Fuk}(W, \Lambda)}(L^-, M) \to \!\lim_{\substack{\longrightarrow \\ M \to M^-}} \Hom_{\mathrm{Fuk}(W, \Lambda)}(L, M^-)) 
\end{eqnarray*}
This recovers a  result originally proven in \cite[Theorem 1.1(2)]{GGV}.
\end{example}

\vspace{2mm}
The remainder of this note concerns the proof of Theorem \ref{rab formula}. 

\vspace{2mm}
We have the diagram: 

$$
\begin{tikzcd}
    \Mod\,\cK \arrow[rr, bend left = 35, "j^{RR}"]   
    \arrow[rr, bend right = 35, "j", swap] 
    & & \Mod\, \cC  \arrow[rr, bend left = 35, "i^R"] \arrow[ll, "j^R", swap]  
    \arrow[rr, bend right = 35, "i^L", swap]
    & &  \Mod\, \cS \arrow[ll, "i", swap] 
\end{tikzcd}
$$

Here, $j^R$ and $i$ are the natural pullback of modules under the identification of ind- and module-categories.  These each have right
and left adjoints, and the left adjoints compose with the Yoneda embeddings
to give the original $j$ and $i^L$:

We note some properties of this diagram. 
The maps $i, j, j^{RR}$ are fully faithful; we have $j^R j = 1_{\Mod\, \cK} = j^R j^{RR}$
and $i^L i = 1_{\Mod\, \cS} = i^R i$. 

\begin{lemma}\label{iRKvanishing}
$i^R(\cK)=0$.
\end{lemma}
\begin{proof}
For $\cE\in \cK$ and $\cF\in \Mod\, \cS$, we have
\begin{equation}
\Hom_{\Mod\, \cS}(\cF, i^R(\cE))\cong \Hom_{\Mod\, \cC}(i(\cF), \cE).
\end{equation}
We have $\cF_i\in \cC$ such that $\displaystyle{\lim_{\substack{\longrightarrow \\ i}}\cF_i=i(\cF)}$. 
\begin{equation}
\begin{split}
\Hom_{\Mod\, \cC}(i(\cF), \cE)&\cong \lim_{\substack{\longleftarrow \\ i}}\Hom_{\Mod\, \cC}(\cF_i, \cE)\\
&\cong \lim_{\substack{\longleftarrow \\ i}}\Hom_{\bK}(\Hom_{\Mod\, \cC}(\Phi^{-1}(\cE), \cF_i),\bK)\\
&\cong \Hom_{\bK}(  \lim_{\substack{\longrightarrow \\ i}}\Hom_{\Mod\, \cC}(\Phi^{-1}(\cE), \cF_i),\bK)\\
&\cong \Hom_{\bK}(\Hom_{\Mod\, \cC}(\Phi^{-1}(\cE), \lim_{\substack{\longrightarrow \\ i}}\cF_i),\bK)\\
&\cong \Hom_{\bK}(\Hom_{\Mod\, \cC}(\Phi^{-1}(\cE), i(\cF)),\bK)=0.
\end{split}
\end{equation}
In the last equality, we used the ansatz $\Phi$ preserves $\cK$. This completes the proof.
\end{proof}

We will later be interested in the Drinfeld-Verdier quotient $(\Mod \, \cC) / \cK$.  (Note this differs from
$\Mod \, \cC / \Mod\, \cK = \Mod \, \cS$.)
It will be useful that certain
morphisms can already be computed in $\cC$:

\begin{lemma}
For any $c, d \in \cC$, 
\begin{equation} \label{indckL_}
    \Hom_{\Mod \cC/\cK}(ii^L(c),d)\cong \Hom_{\Mod \cC}(ii^L(c), d)\cong \Hom_{\Mod \cC}(c,ii^R(d)).
\end{equation}
and
\begin{equation}  \label{indckLL}
    \Hom_{\Mod \cC/\cK}(ii^L(c),ii^L(d))\cong \Hom_{\Mod \cC}(ii^L(c),ii^L(d))\cong \Hom_{\Mod \cC}(c,ii^L(d)).
\end{equation}
Additionally, 
\begin{equation} \label{indck_L}
    \Hom_{\Mod \cC/\cK}(c,ii^L(d))\cong \Hom_{\Mod \cC}(c,ii^L(d))
\end{equation}
and
\begin{equation} \label{indck__}
\Hom_{\Mod \cC/\cK}(c,d)\cong \Hom_{\Mod \cC}(c,ii^L(d)).
\end{equation}
\end{lemma}
\begin{proof}
     A morphism in $\Hom_{\Mod \cC/\cK}(ii^L(c),d)$ is given by a roof diagram
\begin{equation}
ii^L(c)\xrightarrow{f}c'\xleftarrow{g} d
\end{equation}
such that $\Cone(g)\in \cK$. Since $\Hom_{\Mod \cC}(ii^L(c), \Cone(g))=0$ by Lemma~\ref{iRKvanishing}, $f$ is induced by a morphism $ii^L(c)\rightarrow d$. 
This shows \eqref{indckL_}.  Now 
\eqref{indckLL} follows from $i i^R i i^L = i i^L$.  

Similarly, take a morphism in $\Hom_{\Mod \cC/\cK}(c, ii^L(d))$. Then it is given by a roof diagram
\begin{equation}
    c\xleftarrow{f} c'\xrightarrow{g} ii^L(d)
\end{equation}
such that $\Cone(f)\in \cK$. Since $\Hom_{\Mod \cC}(\Cone(f), ii^L(d))=0$, $g$ is induced by a morphism $c\rightarrow ii^L(d)$.
This establishes \eqref{indck_L}.

Finally, 
since $j^R(d)\in \Mod\cK$, we have $d_i\in \cK$ such that $\displaystyle{\lim_{\substack{\longrightarrow \\ i}}}\,d_i=j^R(d)$. Since $j$ is colimit preserving and $c$ is compact, we have
\begin{equation}
    \Hom_{\Mod\cC}(c,jj^R(d))\cong \lim_{\substack{\longrightarrow \\ i}}\Hom_{\Mod\cC}(c,d_i).
\end{equation}
Take any morphism $f\in  \Hom_{\Mod\cC}(c,jj^R(d))$.  The above isomorphism implies  $f$ factors through $d_i\in \cK$ for some sufficiently large $i$. This implies  $\Hom_{\Mod\cC/\cK}(c,jj^R(d))\cong 0$. Applying this result to the triangle
\begin{equation}
    \Hom_{\Mod\cC/\cK}(c,d)\rightarrow  \Hom_{\Mod\cC/\cK}(c,ii^L(d))\rightarrow  \Hom_{\Mod\cC/\cK}(c,jj^R(d))\rightarrow,
\end{equation}
we get \eqref{indck__}.
\end{proof}

\begin{lemma} \label{ppexact}
Given an exact sequence as in \eqref{seq}, 
 the restrictions of $i$ and $j^R$ to pseudo-perfect modules have 
 the following properties: 
\begin{itemize}
\item $i: \cS^{pp} \to \cC^{pp}$ is fully faithful;
\item the image of $i$ is the kernel of 
$j^R$.
\end{itemize} 
\end{lemma}
\begin{proof}
For the second statement:  
$$\cS^{pp} = \Hom(\cS, \Perf(\bK)) = 
\Hom(\cC \oplus_{\cK} 0, \Perf(\bK)) =
\cC^{pp} \times_{\cK^{pp}} 0.$$
\end{proof}

\begin{remark}
Note we do not claim the  map $\cC^{pp}/i(\cS^{pp}) \to \cK^{pp}$  is fully faithful. 
\end{remark}

\begin{corollary} \label{jRkernel} 
Assume $\cC$ is smooth and
proper, so $\cC^{pp} = \cC$.  Then the 
kernel of the map 
$$\cC \xrightarrow{j^R} \cK^{pp} \to \cK^{pp} / \cK$$
is generated by $\cK$ and $\cC \cap i(\cS)$.
\end{corollary}
\begin{proof}
    After Lemma \ref{ppexact}, the only thing remaining to check is $i(\cS^{pp}) = \cC \cap i(\cS)$.
    Smoothness of $\cC$ implies smoothness of $\cS$, 
    hence $\cS^{pp} \subset \cS$, giving 
    the inclusion $\subset$.  
    On the other hand for $s \in \cS$ satisfies 
    $i(s) \in \cC$, then for $c \in \cC$ we have 
    $$\Hom_{\cS}(i^L(c), s) = \Hom_{\cC}(c, i(s))$$
    by properness of $\cC$,  this Hom is perfect.  
    But $i^L$ is surjective, so $s \in \cS^{pp}$. 
\end{proof}


\begin{proof}[Proof of Theorem \ref{rab formula}]
Consider the category $\lb\cC, \Mod\, \cS\rb$ generated by $\cC$ and $\Mod\, \cS$ in $\Mod \,\cC$. Since $j^R$ kills $\Mod\, \cS$, we have an induced functor $\lb \cC, \Mod\,\cS\rb\rightarrow \cK^{pp}$. The kernel is generated by $\Mod \, \cS$, and we have a map 
\begin{equation}
    [j_R]\colon \lb \cC, \Mod\cS\rb/\Mod \cS \to \cK^{pp}.
\end{equation}
As $[j_R]$ can be embedded into an equivalence $\Mod \, \cC/\Mod\, \cS\cong \Mod\cK$, it is in particular
fully faithful.  Hence we get an equivalence:

\begin{equation}
    \lb \lb \cC, \Mod\cS\rb/\Mod \cS\rb/\cK\cong  \widehat \cS_\infty\subset \cK^{pp}/\cK.
\end{equation}

Consider the embedding
\begin{equation}
    \lb \cC, \Mod\cS\rb/\Mod \cS\hookrightarrow \Mod \cK\hookrightarrow\Mod\cC.
\end{equation}
given by $jj^R$.  We use 
the same notation after passing to the quotient by $\cK$: 
\begin{equation}
    jj^R\colon \lb\lb \cC, \Mod\cS\rb/\Mod \cS\rb/\cK\hookrightarrow\Mod\cC/\cK
\end{equation}

Thus far we have shown
$$\Hom_{\widehat \cS_\infty}(c, d)  = 
\Hom_{\Mod \cC/\cK}(jj^R(c), jj^R(d))
$$

Since we have an exact triangle
\begin{equation}
    jj^R \rightarrow\id\rightarrow ii^L\rightarrow,
\end{equation}
we have
\begin{equation}
\begin{split}
    \Hom_{\Mod \cC/\cK}(jj^R(c), jj^R(d))&\cong \Cone( C_1\rightarrow C_2)[-1]
\end{split}
\end{equation}
where
\begin{equation}
    \begin{split}
         C_1&:=\Cone(\Hom_{\Mod \cC/\cK}(ii^L(c),d)\rightarrow\Hom_{\Mod \cC/\cK}(c, d) )\\
         C_2&:=\Cone(\Hom_{\Mod \cC/\cK}(ii^L(c),ii^L(d))\rightarrow \Hom_{\Mod \cC/\cK}(c,ii^L(d))).
    \end{split}
\end{equation}

By \eqref{indckLL}, we see $C_2 = 0$.
To complete the proof we rewrite 
$C_1$ using \eqref{indckL_}  and 
\eqref{indck__}.  
\end{proof}

\appendix

\section{Compositions in $\widehat{S}_\infty$} \label{compositions}

Let $c_0,c_1,c_2$ be objects of $\cC$, viewed also as objects of $\widehat \cS_{\infty}$. We express the underlying complex of $\Hom_{\widehat\cS_\infty}(c_i,c_{i+1})$  as 
\begin{equation}
\Hom_{\Mod\, \cC}(ii^L(c_i), c_{i+1})[1]\oplus\Hom_{\Mod\, \cC}(ii^L(c_i), ii^L(c_{i+1}))).
\end{equation}
We will use the unit morphism
\begin{equation}
    u\colon c_i\rightarrow ii^L(c_i).
\end{equation}

We will compose
\begin{equation}
\begin{split}
(f_0,g_0)&\in \Hom_{\Mod\, \cC}(ii^L(c_0), c_{1})[1]\oplus\Hom_{\Mod\, \cC}(ii^L(c_i), ii^L(c_{i+1})))\\
    (f_1,g_1)&\in \Hom_{\Mod\, \cC}(ii^L(c_1), c_{2})[1]\oplus\Hom_{\Mod\, \cC}(ii^L(c_i), ii^L(c_{i+1}))).
\end{split}
\end{equation}

We use the notation from the proof of Theorem \ref{rab formula}. We have the projection
\begin{equation}
    \pi\colon \Cone(C_1\rightarrow C_2)[-1]\rightarrow C_1,
\end{equation}
which is a quasi-isomorphism. For each $(f_i,g_i)$, we have a cocycle lift
\begin{equation}
\begin{split}
    &(f_i,g_i, u\circ g_i\circ u^{-1}, 0)\\
    &\in \Hom_{\Mod \cC/\cK}(ii^L(c_i),c_{i+1})[-1]\oplus \Hom_{\Mod \cC/\cK}(c_i, c_{i+1})\\
    &\oplus \Hom_{\Mod \cC/\cK}(ii^L(c_i),ii^L(c_{i+1}))[-2]\oplus \Hom_{\Mod \cC/\cK}(c_i,ii^L(c_{i+1})))[-1],
\end{split}
\end{equation}
which is the underlying vector space of $\Cone(C_1\rightarrow C_2)$, which is the underlying vector space of the hom-space $\Hom(\Cone(c_i\rightarrow ii^L(c_i)), \Cone(c_{i+1}\rightarrow ii^L(c_{i+1})))$. Here $g_i\circ u^{-1}$ is only cohomologically well-defined. 
We then directly calculate and get
\begin{equation}
    \begin{split}
        &(f_1,g_1, u\circ g_1\circ u^{-1}, 0)\circ (f_0,g_0, u\circ g_0\circ u^{-1}, 0)\\
        &=(g_1\circ f_0+f_1\circ u\circ g_0\circ u^{-1}, g_1\circ g_0 ,\star_1, \star_2),
    \end{split}
\end{equation}
where the last two components are omitted. 

We interpret each term as a morphism of $\Mod\, \cC$. By taking the following identification, $u^{-1}$ disappears: 
\begin{equation}
\begin{split}
    &(f_i,g_i, g_i, 0)\\
    &\in \Hom_{\Mod \cC}(ii^L(c_i),c_{i+1})[-1]\oplus \Hom_{\Mod \cC}(ii^L(c_i), ii^L(c_{i+1}))\\
    &\oplus \Hom_{\Mod \cC}(ii^L(c_i),ii^L(c_{i+1}))[-2]\oplus \Hom_{\Mod \cC}(ii^L(c_i),ii^L(c_{i+1})))[-1].
\end{split}
\end{equation}
Then the terms in
\begin{equation}
(g_1\circ u\circ f_0+f_1\circ g_0, g_1\circ g_0)
\end{equation}
are well-defined 
except for $g_1\circ f_0$ lands in the correct place $\Hom_{\Mod \cC}(ii^L(c_i),c_{i+1})[-1]\oplus \Hom_{\Mod \cC}(ii^L(c_i), ii^L(c_{i+1}))$. Here we put $u$ the head of two $f_0$, which also comes from the identification with $\Mod\, \cC$. 

A priori, $g_1\circ u\circ f_0$ is not in $\Hom_{\Mod \cC}(ii^L(c_i),c_{i+1})[-1]$, but $\Hom_{\Mod \cC}(ii^L(c_i),ii^L(c_{i+1}))[-1]$. But, by  construction, there is some $u^{-1}\circ g_1\circ u\circ f_0\in\Hom_{\Mod \cC}(ii^L(c_i),c_{i+1})[-1] $ such that $u\circ (u^{-1}\circ g_1\circ f_0)=g_1\circ u\circ f_0$. Hence, at the cohomological level, we obtain the following formula for the composition:
\begin{equation}
    (f_1, g_1)\circ (f_0,g_0):=(u^{-1}\circ g_1\circ u\circ f_0+f_1\circ g_0, g_1\circ g_0).
\end{equation}

One way to write formulas beyond the cohomological level would be the following.  Choose a projection $C_1\rightarrow H^*(C_1)$ and the splitting of $\Cone(C_1\rightarrow C_2)[1]\rightarrow H^*(C_1)$, one obtains the contracting homotopy from $\Cone(C_1\rightarrow C_2)[1]$ to $H^*(C_1)$. Then, by running the homological perturbation theory, one obtains an $A_\infty$-structure upgrading the above composition formula, which is by construction quasi-equivalent to $\widehat{\cS}_\infty$.

\section{$\Hom_{\bK[t]}(\bK[t, t^{-1}], \bK[t])$} \label{ridiculous hom calculation}

 A free resolution of $\bK[t, t^{-1}]$
 is given by: 
\begin{eqnarray*}
 \bigoplus_{n \leq -1} \bK[t] \cdot r_n & \to &  \bigoplus_{n \leq 0} \bK[t] \cdot s_n \\
    r_n & \mapsto & t s_n - s_{n+1} \\
\end{eqnarray*} 
where $r_n, s_n$ are just basis elements.  Dualizing gives 
\begin{eqnarray*}
    \prod_{n \leq 0} \bK[t] \cdot s_n^* &\to &     \prod_{n \leq -1} \bK[t] \cdot r_n^* \\
    s_n^* & \mapsto & t r_n^*-  r_{n-1}^* 
\end{eqnarray*}

Consider the following $\bK[t]$-linear map
\begin{equation}
    \Sigma\colon \prod_{n\leq -1}\bK[t]r_n^*\rightarrow \bK[[t]]; r_n^*\mapsto t^{-n-1}.
\end{equation}
We claim that 
\begin{equation}
\prod_{n \leq -1} \bK[t] \cdot s_n^* \to     \prod_{n \leq -1} \bK[t] \cdot r_n^* \rightarrow \bK[[t]]\rightarrow 0
\end{equation}
is an exact sequence.  Indeed, it is obvious that the composition is zero. Suppose $\prod f_n(t)r_n^*$ goes to zero. For each monomial $\alpha r_n^*$  of $\prod f_n(t)r_n^*$, we set $$\deg (\alpha r_n^*):=\deg (\alpha)-n-1.$$
Let $N$ be the lowest nonzero number where $\prod f_n(t)r_n^*$ has a nonzero degree $N$ monomial. Note that the number of degree $N$ monomials in $\prod f_n(t)r_n^*$ are finite. Hence, by adding an element coming from $\prod_{n \leq -1} \bK[t] \cdot s_n^*$, one can assume that the sum of the degree $N$ monomials is $\beta r_{-N-1}^*$ for some scalar $\beta$. Since this is still in the kernel of $\Sigma$ and the degree $N$-part of $\Sigma(\beta r_{-N-1}^*)=\beta t^{N}$, $\beta$ is zero. Inductively, adding elements coming from $\prod_{n \leq -1} \bK[t] \cdot s_n^*$, we get $\ker \Sigma=\prod_{n \leq -1} \bK[t] \cdot s_n^*$.

Hence 
\begin{equation}
\prod_{n \leq 0} \bK[t] \cdot s_n^* \to     \prod_{n \leq -1} \bK[t] \cdot r_n^* \rightarrow \bK[[t]]/\bK[t]\rightarrow 0
\end{equation}
is also an exact sequence.  
(It is also easy to see that the first map is injective.) 

\subsection*{Acknowledgment}
We would like to thank Adrian Petr
for some questions 
about the Rabinowitz Fukaya category. We would like to thank the anonymous referees for pointing out a missing assumption in the main theorem.
The first-named author's work is supported by JSPS KAKENHI Grant Numbers 22K13912, 20H01794, 	23H01068.
The second-named author's work is 
supported Novo Nordisk Foundation grant NNF20OC0066298, 
Villum Fonden Villum Investigator grant 37814, and Danish National Research Foundation grant DNRF157.

\bibliographystyle{plain}
\bibliography{bibs}

\end{document}